\documentclass[11pt]{amsart}

\usepackage{tikz}
\tikzstyle{every node}=[circle, draw, fill=black!50,
                        inner sep=0pt, minimum width=3pt]

\usepackage{amsfonts,amsmath,latexsym,color,epsfig,hyperref,enumitem, amssymb,bbm}

\AtBeginDocument{\addtocontents{toc}{\protect\setlength{\parskip}{0pt}}}

\newtheorem{theorem}{Theorem}[section]
\newtheorem{lemma}{Lemma}[section]

\newtheorem{cor}[lemma]{Corollary}

\newtheorem{obs}[lemma]{Observation}

\renewcommand{\le}{\leqslant}
\renewcommand{\ge}{\geqslant}

\def\qed{\ifvmode\mbox{ }\else\unskip\fi\hskip 1em plus 10fill$\Box$}
\input{epsf}

\makeatletter
\def\Ddots{\mathinner{\mkern1mu\raise\p@
\vbox{\kern7\p@\hbox{.}}\mkern2mu
\raise4\p@\hbox{.}\mkern2mu\raise7\p@\hbox{.}\mkern1mu}}
\makeatother

\title{An isoperimetric inequality for word overlap}
\author{Dmitrii Zakharov}

\address{Department of Mathematics, Massachusetts Institute of Technology, Cambridge, MA 02139, USA}
\email{zakhdm@mit.edu}

\date{}

\begin{document}

\begin{abstract}
Let $A$ and $B$ be sets of words of length $n$ over some finite alphabet. Suppose that no suffix of a word in $A$ coincides with a prefix of a word in $B$. Then we show that the product of densities of $A$ and $B$ is upper bounded by $(1+o(1))/(en)$. This bound is asymptotically sharp.
\end{abstract}

\maketitle

\section{Introduction}

Let $\Omega$ be a finite set and let $n\ge 2$. Let $\mu=\mu_n$ be the uniform probability measure on $\Omega^n$. We say that an ordered pair of words $(w, u) \in \Omega^n \times \Omega^n$ {\em overlaps} if a final segment of $w$ coincides with an initial segment of $u$. That is, if we denote $w=(w_1, \ldots, w_n)$, $u=(u_1, \ldots, u_n)$ then for some $j\in \{1, \ldots, n\}$ we have $(w_{n-j+1}, \ldots, w_n) = (u_1, \ldots, u_j)$. Note that we in particular allow $u = w$. 
We are interested in the following extremal question: suppose that $A, B\subset \Omega^n$ are sets of words such that no two words $w \in A$ and $u \in B$ overlap. For what pairs of densities $\alpha, \beta \in (0,1)$ is it possible to have $\mu(A) \ge \alpha$ and $\mu(B) \ge \beta$?

There is a related question about non-overlapping codes (also known as `cross-bifix-free' codes) that has been extensively studied in the computer science literature \cite{bernini2017gray, blackburn2015non, chee2013cross, levenshteindecoding, stanovnik2024search}. In our notation, the question is to determine the size of a largest code $A \subset \Omega^n$ such that no two distinct words in $A$ overlap (see \cite{chee2013cross} for an asymptotically sharp construction). So the question we consider can be thought of as a bipartite variant of this and, to the best of our knowledge, it has not been studied before.

Define the shift map $s=s_n: \Omega^n \rightarrow \Omega^{n-1}$ by
\[
s(w_1, \ldots, w_n) = (w_2, \ldots, w_n).
\]
For a subset $A \subset \Omega^n$ we can define the set of words which do not overlap with $A$ as follows:
\[
U= U(A) = \Omega^n \setminus \bigcup_{j=0}^{n-1} s^j(A) \times \Omega^{j}.
\]
It is easy to see that $U(A)$ is precisely the set of all $u$ such that the pair $(w,u)$ does not overlap for all $w\in A$. 
Let $\gamma(\alpha, n)$ be the largest possible measure of the set $U(A)$ over all $A\subset \Omega^n $ of measure $\alpha$ and all finite sets $\Omega$. Here we consider the uniform measure on the space $\Omega^n$. 
Note that if $A \subset A'$ then we have the inclusion $U(A') \subset U(A)$. This means that 
$\gamma(\alpha, n)$ is a monotone decreasing function in $\alpha$. For example, it is an easy exercise to show that $\gamma(\alpha,2) = \max\{ (1-\alpha)^2, ~1-\alpha^{1/2} \}$ holds for any $\alpha \in (0,1)$.

Our result is the following estimate.

\begin{theorem}\label{thm:isopery}
We have $\gamma(\alpha, n) \le \big(\frac{n}{n+1}\big)^{n+1} \cdot \frac{1}{\alpha n}$ for all $n\ge 1$ and $\alpha \in (0, 1)$.
\end{theorem}

So if we have a pair of non-overlapping sets $A, B\subset \Omega^n$ with densities $\alpha$ and $\beta$ then we have $\alpha\beta  n\le \big(\frac{n}{n+1}\big)^{n} = (1/e +o(1)) $. 

The proof of Theorem \ref{thm:isopery} is presented in the next section. The rough idea is to split the set $\Omega^n \setminus U(A)$ into several disjoint pieces and use inclusion-exclusion to lower bound the size of each piece. This then gives a certain recursive relationship between various densities associated with $A$ and $U(A)$ and their shifts. A careful algebraic manipulation completes the proof.

We close this section by considering some examples essentially matching the upper bound in Theorem \ref{thm:isopery}. Let $S \subset \Omega$ be an arbitrary subset.
Then for $A = S^n$ one can check that
\[
U(A) = (\Omega\setminus S)\times \Omega^{n-1}
\]
and so we get
\[
\mu(U(A)) = 1-\alpha^{1/n} = \frac{\log(1/\alpha)}{n} + O\left(\frac{\log^2(1/\alpha)}{n^2}\right).
\]
This is a good bound for $\alpha \in (1/2, 1)$. 
    
Similarly, for $A = \Omega^{n-1} \times S$ one can check that
\[
U(A) = (\Omega\setminus S)^n
\]
so that we get
\[
\mu(U(A)) = (1-\alpha)^n.
\]
This is a good bound for $\alpha \in (0, 1/n)$. 

We can interpolate between these two examples by taking $A = \Omega^{n-k} \times S^k$ for some $1 \le k \le n$.
Then the set $U = U(A)$ is given by
\begin{align*}
U =  \{ (w_1, \ldots, w_n):~ w_1 \not\in S, \quad \{w_{j+1}, \ldots, w_{j+k}\}\not\subset S,~j=0, \ldots, n-k \}.    
\end{align*}
The exact formula for $\mu(U)$ is a bit complicated (it involves generalized Fibonacci numbers, see \cite{chee2013cross}) but we can use a simple Poisson approximation inequality due to \cite{arratia1989two} (see also \cite{godbole1991poisson, godbole1993improved, guibas1980long}) to get a good estimate of the measure of $U$.
Denote $p = |S|/|\Omega|$ so that $\alpha = p^k$. 
Let $Z_1, \ldots, Z_{n-1} \sim \operatorname{Ber}(p)$ be iid Bernoulli random variables. Let $R_{n-1}$ be the length of the longest run of 1-s in the sequence $(Z_1, \ldots, Z_{n-1})$. The measure of $U$ can then be computed in terms of $R_{n-1}$:
\[
\mu(U) = (1-p)\Pr[ R_{n-1} < k].
\]
Indeed, we can view $w_2, \ldots, w_n \in \Omega$ as iid variables uniformly distributed on $\Omega$ and select $Z_i = 1_{w_{i+1} \in S}$. 
By \cite[Example 3]{arratia1989two} we have the following estimate on this probability:
\[
\left|\Pr[ R_{n-1} < k] - e^{-\lambda}\right| \le \frac{\lambda (2k+1)}{n-1} + 2 p^k, \quad \lambda = p^k ((n-2)(1-p)+1)
\]
Let $k = [n\alpha \log(1/\alpha) ]$. Then for $1/n \ll \alpha \ll 1$ we have $p = \alpha^{1/k} = e^{-\frac{\log(1/\alpha)}{k}} = 1 - \frac{1+o(1)}{n\alpha}$. This gives $\lambda = 1+o(1)$ and $\Pr[R_{n-1} < k] = e^{-1} + o(1)$ and so we have
\[
\mu(U) = \frac{e^{-1}+o(1)}{\alpha n},
\] 
where $o(1)$ tends to zero as $\min(\alpha^{-1}, n\alpha) \to \infty$. This matches the bound in Theorem \ref{thm:isopery} for all $\alpha \in (1/n, 1/2)$ up to a $(1+o(1))$ factor. 

\subsection*{Statement of AI use}
The original version of the paper derived a bound of the form $(1+o(1))/n$ for Theorem \ref{thm:isopery}. A much simpler proof of an improved upper bound of the form $1/2n$ was given shortly after this paper was posted to arXiv by combining an argument of GPT 5.2 Pro, at the direction of Dmitry Rybin, with an additional twist by Nikita Gladkov. Then an internal model at OpenAI, at the direction of Yuzhou Gu, was the able to successfully derive the current sharp version of Theorem \ref{thm:isopery} (given just the target statement). Upon reading this argument, the original author realized that the improvement was coming from an improved analysis of the inequality given in Section \ref{eq:gamma}. This improved analysis itself can also be derived via using GPT 5.2 Pro; \url{https://chatgpt.com/share/69a705a7-c100-8009-9314-46a83197adb7} contains the transcript with the author.

\section{Proof of Theorem \ref{thm:isopery} and a corollary}

We will repeatedly use the following simple observation. For $A \subset \Omega^n$, $r\le n$ and $w \in \Omega^{r}$ we denote $A(w) = \{u\in \Omega^{n-r}:~(w,u) \in A\}$.

\begin{obs}\label{obs1}
    Let $A \subset \Omega^n$ and $B \subset \Omega^r$ for some $r \le n$. Then we have 
    \[
    \mu(A \cap (B \times \Omega^{n-r})) \le \lambda \mu(B),
    \]
    where $\lambda = \max_{w\in \Omega^r} \mu(A(w))$. 
\end{obs}

Indeed, we simply apply the definition of $\lambda$ for each $w \in B$ and sum over. 
For $j\le n$ denote $A_j = s^{n-j}(A) \subset \Omega^{j}$ and let $\alpha_j = \mu(A_j)$, $j=1, \ldots, n$. For $1\le r \le j-1$ define
\[
\lambda_{j, r} = \max_{w \in \Omega^{r}} \mu(A_j(w)).
\]
Note that for any $w \in \Omega^r$ we have $A_{j}(w) \subset s^{r}(A_j) = A_{j-r}$. This implies that $\alpha_{j-r} \ge \lambda_{j, r}$.

Denote 
\[
B_j = \bigcup_{i=1}^{j} A_{i} \times \Omega^{j-i} \subset \Omega^{j}
\]
and denote $\beta_j = \mu(B_j)$. By the definition of $U$ it follows that $B_j = \Omega^j \setminus U(A_j)$.

We trivially have $\beta_{1}=\alpha_{1}$. Since $A_{j} = s(A_{j+1})$, we have the inclusion $A_{j+1} \subset \Omega\times A_{j} $. We have $B_{j+1} = (B_j \times \Omega) \cup A_{j+1}$ and, in particular $B_j \times \Omega \subset B_{j+1}$. Together these observations imply the following chain of inequalities:
\[
\beta_n \ge \ldots \ge \beta_1 = \alpha_1 \ge \ldots \ge \alpha_n.
\]
Now let us define sets $D_j$ as follows:
\[
D_j = A_j \setminus (B_{j-1} \times \Omega) = B_j \setminus (B_{j-1} \times \Omega),
\]
where for $j=1$ we put $D_1 = A_1=B_1$. 
In particular, since $B_{j-1} \times \Omega \subset B_{j}$, we can write $B_j$ as a disjoint union $(B_{j-1}\times \Omega) \sqcup D_j$ and
$\mu(D_j) = \beta_j-\beta_{j-1}$ for all $j=1, \ldots, n$ (where we set $\beta_0=0$).
Note that we can write 
\begin{align*}
B_j &= D_j \sqcup (B_{j-1}\times \Omega) = D_j \sqcup (D_{j-1}\times \Omega) \sqcup (B_{j-2}\times \Omega^2) = \ldots     \\
&= D_j \sqcup (D_{j-1}\times \Omega) \sqcup (D_{j-2}\times \Omega^2) \sqcup \ldots \sqcup (D_{1}\times \Omega^{j-1}) \\
& = A_j \cup \left( (D_{j-1}\times \Omega) \sqcup (D_{j-2}\times \Omega^2) \sqcup \ldots \sqcup (D_{1}\times \Omega^{j-1}) \right).
\end{align*}
Using Observation \ref{obs1}, we have the following bounds for $i=1, \ldots, j-1$:
\[
\mu(A_j \cap (D_i \times \Omega^{j-i})) \le \lambda_{j, i}\cdot  \mu(D_i) \le \alpha_{j-i} \cdot \mu(D_i)
\]
So since sets $D_i \times \Omega^{j-i}$ are pairwise disjoint, we obtain
\[
\mu(B_j) = \mu(A_j) + \sum_{i=1}^{j-1} \mu(D_i \times \Omega^{j-i} \setminus A_j) \ge \mu(A_j) + \sum_{i=1}^{j-1} (1- \alpha_{j- i}) \mu(D_{i})
\]
giving the following relation between $\alpha$-s and $\beta$-s:
\begin{equation}\label{eq:beta}
    \beta_j \ge \alpha_j + \sum_{i=1}^{j-1} (1-\alpha_{j-i}) (\beta_{i}-\beta_{i-1}).
\end{equation}
Denote $\gamma_i = 1-\beta_i = \mu(U(A_i))$ and let $\delta_i = \alpha_{i-1}-\alpha_i$ for $i=1, \ldots, n$ where we put $\alpha_0=1$ and $\gamma_0=1$. Then (\ref{eq:beta}) can be rewritten as follows:
\begin{equation}\label{eq:gamma}
    \gamma_j \le \sum_{i=0}^{j-1} \gamma_i\delta_{j-i}.
\end{equation}
We also have the following information about $\gamma_i, \delta_i$: 
\[
\gamma_n \le \gamma_{n-1} \le \ldots \le \gamma_1 \le \gamma_0 = 1,
\]
\[
\delta_1+\ldots+\delta_n = 1-\alpha, \quad \delta_i \ge 0, \quad i=1, \ldots, n.
\]
We will use these properties to upper bound $\gamma_n$. Rewrite (\ref{eq:gamma}) as
\[
\delta_j \ge \gamma_j - \sum_{i=1}^{j-1} \gamma_i \delta_{j-i}.
\]
For $\rho >0$ multiply both sides by $\rho^{j-1}$ and sum over $j=1, \ldots, n$:
\begin{align}
\sum_{j=1}^n \rho^{j-1} \delta_j \ge& \sum_{j=1}^n \rho^{j-1}\left(\gamma_j - \sum_{i=1}^{j-1} \gamma_i \delta_{j-i}\right)  \nonumber \\
=& \sum_{j=1}^n \rho^{j-1} \gamma_j \left( 1 - \sum_{i=1}^{n-j} \rho^i \delta_{i} \right) \label{eq:reaarange}
\end{align}
Now let us define $F(\rho) = \sum_{j=1}^n \delta_j \rho^j $. Notice that $F(1) = \delta_1+\ldots+\delta_n = 1-\alpha < 1$ and $F$ is strictly increasing on $[1, \infty)$. Thus, there exists a unique $\rho > 1$ such that $F(\rho)=1$. For this $\rho$, all brackets on the right hand side of (\ref{eq:reaarange}) are non-negative. So using $\gamma_n \le \gamma_j$ we obtain
\begin{align*}
    \rho^{-1} = \rho^{-1}F(\rho)&=\sum_{j=1}^n \rho^{j-1} \delta_j \ge \gamma_n \sum_{j=1}^n   \rho^{j-1} \left( 1 - \sum_{i=1}^{n-j} \rho^i \delta_{i} \right) \\
    & = \gamma_n \left( \frac{\rho^n-1}{\rho-1} - \sum_{i=1}^n \delta_i \frac{\rho^{n}-\rho^i}{\rho-1} \right) \\
    & = \gamma_n \frac{ \alpha \rho^n + F(\rho) -1 }{\rho-1} = \alpha \gamma_n \frac{\rho^n}{\rho-1}.
\end{align*}
We conclude that $\alpha \gamma_n \le \frac{\rho-1}{\rho^{n+1}}$ for some $\rho >1$. By computing the derivative, one can verify that the right hand side is maximized at $\rho =\frac{n+1}{n}$. We conclude that $\alpha \gamma_n \le \frac{1}{n} \left(\frac{n}{n+1}\right)^{n+1} \le \frac{1}{en}$, concluding the proof. 




We have the following corollary of Theorem \ref{thm:isopery}, which gives a `small level set' estimate for the union of shift sets $s^j(A) \times \Omega^{n-j}$.

\begin{cor}\label{cor:isoperimetry}
Let $A \subset \Omega^n$ be a set of measure $\alpha \in (0,1)$. For $w \in \Omega^n$ define $f(w)$ to be the number of indices $j\in \{0, \ldots, n-1\}$ such that $w \in s^j(A) \times \Omega^{j}$.
Then for every integer $t \in [1, n/4]$ we have the following level set estimate on $f$:
\begin{equation}
    \mu(\{w\in \Omega^n:~ f(w) \le t\}) \le \frac{4t}{\alpha n}.
\end{equation}
\end{cor}

For example, by taking $t = \alpha n/8$ and assuming that $\alpha \ge 1/n$, we obtain that at least half of elements $w \in \Omega^n$ is covered by at least $\alpha n/8$ many sets of the form $s^j(A) \times \Omega^j$. Taking $t=1$ on the other hand recovers Theorem \ref{thm:isopery} in the range $\alpha\in (1/n, 1/2)$, albeit with a constant factor loss.

\begin{proof}
    If $t\ge \alpha n/2$ then there is nothing to prove, so we may assume $t\le \alpha n/2$ holds.

    Let $\tilde n = [n/2t]$ and $r = n-2t \tilde n$. By the assumption on $t$ we have $\tilde n\ge 1$. Consider a new alphabet $\tilde \Omega = \Omega^{2t}$ and
    let $\tilde s = \tilde s_j: \tilde \Omega^{j} \rightarrow \tilde \Omega^{j-1}$ denote the shift map defined on words over the alphabet $\tilde \Omega$. By identifying $\tilde \Omega^j = \Omega^{2t j}$, we get that $\tilde s = s^{2t}$. For $i=0, \ldots, 2t-1$ let 
    \[
    \tilde A_i = s^{r+i}(A) \times \Omega^{i} \subset \Omega^{2t \tilde n} = \tilde \Omega^{\tilde n}.
    \]
    Note that $\mu(\tilde A_i) \ge \mu(A) =\alpha$ for $i=0, \ldots, 2t-1$. For an arbitrary subset $\tilde A \subset \tilde \Omega^{\tilde n}$ we denote $\tilde U(\tilde A) = \tilde \Omega^{\tilde n} \setminus \bigcup_{j=0}^{\tilde n-1} \tilde s^j(\tilde A) \times \tilde \Omega^{j}$, that is the analogue of $U(A)$ over the new alphabet.
    Let $w \in \Omega^n$ and denote $\tilde w = s^r(w) \in \tilde\Omega^{\tilde n}$.
    Note that we have $\tilde w \not \in \tilde U(\tilde A_i)$ precisely when there exists $j \in \{0, \ldots, \tilde n-1\}$ such that $\tilde w \in \tilde s^j(\tilde A_i) \times \tilde \Omega^j$. The latter is in turn equivalent to $w \in s^{2t j + i + r}(A)$. It follows that we have
    \[
    \#\{i \in \{0, \ldots, 2t-1\}: \tilde w \not \in \tilde U(\tilde A_i)\} \le  \#\{ i\in \{0, \ldots, n-1\}: w \in s^{i}(A)\times \Omega^i \}  = f(w)
    \]
    Thus, if $f(w) \le t$ then there are at least $2t-t=t$ indices $i \in \{0, \ldots, 2t-1\}$ such that $\tilde w \in \tilde U(\tilde A_i)$. So by the union bound and Theorem \ref{thm:isopery} applied to each $\tilde U(\tilde A_i)$ we have 
    \[
    t\mu(\{w:~f(w) \le t\}) \le \sum_{i=0}^{2t-1} \mu(\tilde U(\tilde A_i)) \le 2t \cdot   \frac{(n/(n+1))^{n}}{\alpha \tilde n} \le \frac{t}{\alpha \tilde n}.
    \]
    So recalling that $\tilde n = [n / 2t]$ we get
    \[
    \mu(\{w:~f(w) \le t\}) \le \frac{1}{\alpha [n/2t]} \le \frac{4t}{\alpha n}
    \]
    provided that $n \ge 4t$, concluding the proof.
\end{proof}

\bibliographystyle{amsplain0.bst}
\bibliography{main}

@article{chee2013cross,
  title={Cross-bifix-free codes within a constant factor of optimality},
  author={Chee, Yeow Meng and Kiah, Han Mao and Purkayastha, Punarbasu and Wang, Chengmin},
  journal={IEEE Transactions on Information Theory},
  volume={59},
  number={7},
  pages={4668--4674},
  year={2013},
  publisher={IEEE}
}

@article{arratia1989two,
  title={Two moments suffice for Poisson approximations: the Chen-Stein method},
  author={Arratia, Richard and Goldstein, Larry and Gordon, Louis},
  journal={The Annals of Probability},
  pages={9--25},
  year={1989},
  publisher={JSTOR}
}

@article{blackburn2015non,
  title={Non-overlapping codes},
  author={Blackburn, Simon R},
  journal={IEEE Transactions on Information Theory},
  volume={61},
  number={9},
  pages={4890--4894},
  year={2015},
  publisher={IEEE}
}

@article{stanovnik2024search,
  title={In search of maximum non-overlapping codes},
  author={Stanovnik, Lidija and Mo{\v{s}}kon, Miha and Mraz, Miha},
  journal={Designs, codes and cryptography},
  volume={92},
  number={5},
  pages={1299--1326},
  year={2024},
  publisher={Springer}
}

@misc{levenshteindecoding,
  title={Decoding automata, invariant with respect to the initial state. Problemy Kibernet. 12 (1964), 125-136},
  author={Levenshtein, VI},
  publisher={Russian}
}

@article{bernini2017gray,
  title={A Gray code for cross-bifix-free sets},
  author={Bernini, Antonio and Bilotta, Stefano and Pinzani, Renzo and Vajnovszki, Vincent},
  journal={Mathematical Structures in Computer Science},
  volume={27},
  number={2},
  pages={184--196},
  year={2017},
  publisher={Cambridge University Press}
}

@article{guibas1980long,
  title={Long repetitive patterns in random sequences},
  author={Guibas, LJ and Odlyzko, AM},
  journal={Zeitschrift f{\"u}r Wahrscheinlichkeitstheorie und verwandte Gebiete},
  volume={53},
  number={3},
  pages={241--262},
  year={1980},
  publisher={Springer}
}

@article{godbole1991poisson,
  title={Poisson approximations for runs and patterns of rare events},
  author={Godbole, Anant P},
  journal={Advances in applied probability},
  volume={23},
  number={4},
  pages={851--865},
  year={1991},
  publisher={Cambridge University Press}
}

@article{godbole1993improved,
  title={Improved Poisson approximations for word patterns},
  author={Godbole, Anant P and Schaffner, Andrew A},
  journal={Advances in applied probability},
  volume={25},
  number={2},
  pages={334--347},
  year={1993},
  publisher={Cambridge University Press}
}

\end{document}